\documentclass[12pt]{amsart}
\usepackage[colorlinks=true,citecolor=black,linkcolor=black,urlcolor=blue]{hyperref}
\usepackage{amsmath,cleveref}
\usepackage[english,  activeacute]{babel}
\usepackage[utf8]{inputenc}
\usepackage{longtable}
\usepackage{amssymb}
\usepackage{amsthm}
\usepackage{enumitem}
\setitemize{leftmargin=.15in}
\usepackage{graphics,graphicx}

\usepackage{array}
 
\usepackage{bm}

\usepackage{a4wide}
\usepackage{color, url}
\usepackage{float}
\usepackage{mathrsfs}
\usepackage{comment}
\usepackage{xcolor}
\usepackage{relsize}
\allowdisplaybreaks
\usepackage{tikz}
\usetikzlibrary{cd}
\usetikzlibrary{matrix,arrows}

\newcommand{\seqnum}[1]{\href{http://oeis.org/#1}{\underline{#1}}}

\newcommand{\tr}[1]{\textcolor{red}{#1}}

\theoremstyle{plain}
\newtheorem{theorem}{Theorem}[section]
\newtheorem{proposition}[theorem]{Proposition}
\newtheorem{corollary}[theorem]{Corollary}
\newtheorem{lemma}[theorem]{Lemma}

\theoremstyle{definition}
\newtheorem{definition}[theorem]{Definition}
\newtheorem{example}[theorem]{Example}

\newtheorem{remark}[theorem]{Remark}

\newcommand{\PF}{\mathrm{PF}}

\newcommand{\FR}{\mathrm{FR}}

\newcommand{\FRTypeOne}{\mathrm{FR}^{\mathrm{T1}}}
\newcommand{\FRTypeTwo}{\mathrm{FR}^{\mathrm{T2}}}
\newcommand{\FRTypeThree}{\mathrm{FR}^{\mathrm{T3}}}
\newcommand{\UPF}{\mathrm{UPF}}
\newcommand{\UPFTypeOne}{\mathrm{UPF}^{\mathrm{T1}}}
\newcommand{\UPFTypeTwo}{\mathrm{UPF}^{\mathrm{T2}}}
\newcommand{\UPFTypeThree}{\mathrm{UPF}^{\mathrm{T3}}}
\newcommand{\Fub}{\mathrm{Fub}}
\newcommand{\Z}{{\mathbb Z}}

\newcommand{\B}{{\mathcal B}}

\def\modd#1 #2{#1\ \mbox{\rm (mod}\ #2\mbox{\rm )}}

\def\sts#1#2{S_2\left(#1, #2\right)}

\newcommand{\floor}[1]{\lfloor #1 \rfloor}

\usepackage[usestackEOL]{stackengine}[2013-10-15]

\stackMath

\title{Restricted Fubini Rankings and Restricted Unit Interval Parking Functions}

\date{July 11, 2025}

\author[C. Barreto]{Camilo Barreto}

\author[P.~E.~Harris]{Pamela E. Harris}
\address[P.~E.~Harris]{Department of Mathematical Sciences, University of Wisconsin-Milwaukee, Milwaukee, WI 53211 United States}
\email{\textcolor{blue}{\href{mailto:peharris@uwm.edu}{peharris@uwm.edu}}}

\author[J. L. Ram\'{\i}rez]{Jos\'e L. Ram\'{\i}rez}

\author[S. Ram\'{\i}rez]{Samuel Ram\'{\i}rez}

\author[J. C. V\'{a}squez]{Julio C.  Vasquez}
\address[C. Barreto, J. L. Ram\'{\i}rez, S. Ram\'{\i}rez, J. C. V\'{a}squez]{Departamento de Matem\'aticas,  Universidad Nacional de Colombia,  Bogot\'a, Colombia}
\email{\textcolor{blue}{\href{mailto:cbarreto@unal.edu.co}{cbarreto@unal.edu.co}}}
\email{\textcolor{blue}{\href{mailto:jlramirezr@unal.edu.co}{jlramirezr@unal.edu.co}}}
\email{\textcolor{blue}{\href{mailto:samramirezra@unal.edu.co}{samramirezra@unal.edu.co}}}
\email{\textcolor{blue}{\href{mailto:julcvasquezare@unal.edu.co}{julcvasquezare@unal.edu.co}}}

\begin{document}
\begin{abstract}
We study three natural types of restrictions on Fubini rankings and unit interval parking functions, which are motivated by 
their correspondence with ordered set partitions. For each restriction type, we define the corresponding subset of Fubini rankings and unit interval parking functions, establish enumerative results, and provide bijections between the restricted families. 
We also obtain exponential generating functions and combinatorial interpretations, including connections with exceedances in permutations and with the absence of cyclical adjacencies in set partitions.
\end{abstract}

\maketitle

\section{Introduction}
\label{intro}
Throughout, we let $n\in\mathbb{Z}^+=\{1,2,3,\ldots\}$ and $[n]=\{1,2,\ldots,n\}$. Parking functions of length $n$ are a set of tuples encoding the parking preference of cars parking on a one-way street that allow all cars to park on the $n$ spots on the street. 
More precisely, a tuple $(a_1,a_2,\ldots,a_n)\in[n]^n$ has $a_i$ as the parking preference of car $i$, for all $1\leq i\leq n$. Cars enter the street in order and parks in the first spot it encounters at or past its preference. If all cars are able to park under this parking scheme, then the tuple is called a parking function. 
We let $\PF_n$ denote the set of parking functions and recall that these were introduced by Konheim and Weiss~\cite{konheim1966occupancy} who established that $|\PF_n|=(n+1)^{n-1}$.

Since their introduction, parking functions have received a lot of attention, given their numerous connections to a variety of areas. 
For example, parking functions have connections to hyperplane arrangements~\cite{partial,stanleyhyperplane}, computing volumes of flow polytopes~\cite{Benedetti2018ACM}, diagonal harmonins~\cite{diagonal}, counting Boolean intervals in the Weak order of the symmetric group~\cite{CountingBoolean}, the Tower of Hanoi~\cite{displacement_hanoi}, and the \texttt{Quicksort} algorithm~\cite{quicksort}, among many such connections.
Generalizations of parking functions abound and include cases where the cars are able to back up~\cite{knaple,knaples2}, cars are able to bump cars out of their parking spots~\cite{MVP}, where the cars have a variety of lengths~\cite{inv_assort,assortandseq}, where the cars have a fixed interval in which they tolerate parking~\cite{ellinterval,fang2024vacillatingparkingfunctions,Hadaway_unit_interval,subset}, and some in which certain spots are unavailable at the beginning of parking~\cite{completions}. 
There is also work on studying discrete statistics of parking functions to include counting the number of lucky cars, those cars that park in their preference~\cite{colmenarejo2024luckydisplacementstatisticsstirling,FixedLuckyCars}, and counting ties, descents, ascents, and fixed points~\cite{statsofPFs,Schumacher_Descents_PF,fixedpointsofPFs}.
For more on parking functions, we recommend~\cite{yan2015parking} and for an article with open problems see~\cite{choose}.

Among subsets of parking functions, unit interval parking functions have been shown to have a rich combinatorial structure. 
We recall that unit interval parking, introduced by Hadaway~\cite{Hadaway_unit_interval}, are the subset of parking functions in which a car parks either in its preference of in the spot after its preference. We let $\UPF_n$ denote the set of unit interval parking functions of length $n$. 
Hadaway established a bijection between unit interval parking functions and Fubini rankings, which are tuples encoding the results of a competition of $n$ players in which ties are allowed. 
We let $\FR_n$ denote the set of Fubini rankings with $n$ competitors. 
Fubini rankings are in bijection with ordered set partitions, where 
a set partition of $[n]$ is a collection of pairwise disjoint subsets, which we call
blocks, whose union is precisely $[n]$, while an ordered set partition of $[n]$ is simply a set partition in which the order of the blocks matter. 
The number of partitions of $[n]$ having exactly $k$ blocks is the Stirling numbers of the second kind  \seqnum{A008277}, and the number of ordered set partitions of $[n]$ having $k$ blocks is the Fubini numbers (also known as the ordered Bell numbers) \seqnum{A000670}.
Given the aforementioned bijections, the number of unit interval parking functions and the number of Fubini rankings is a Fubini number. 

In this article, we are interested in the study of $S$-restricted Fubini rankings and $S$-restricted unit interval parking functions. 
This is motivated by the connection of Fubini rankings and unit interval parking functions to ordered set partitions, and  by the study of $S$-restricted set partitions. 
B\'enyi and Ram\'irez~\cite{Benyi} studied set partitions with the additional restriction
that the size of each block is contained in a given subset $S \subseteq \mathbb{Z}^+$ and they call these $S$-restricted set partitions. For more on restricted set partitions see \cite{Mendez, Benyi, Caicedo, EGS, Miksa, Moll}.

We introduce three types of restricted Fubini rankings and three types of restricted unit interval parking functions. We define those next and follow that with a list of our main results.
\begin{definition}\label{def:all defs together}
Let $S\subseteq\mathbb{Z}^+$ be a subset of positive integers. 
\begin{center}
\begin{longtable}{lp{.82\textwidth}}
$\UPF$ Type 1:&
We define $\UPFTypeOne_{n,S}$  as the set of unit interval parking functions of length $n$ such that the number of lucky cars belongs to $S$. The set $\UPFTypeOne_{n,S}$, is called the set of \emph{$S$-restricted unit interval parking functions of type 1}.\\
$\FR$ Type 1:&
We define $\FRTypeOne_{n,S}$ as the set of Fubini rankings with $n$ competitors such that the number of distinct ranks belongs to $S$.
The set $\FRTypeOne_{n,S}$, is called the set of \emph{$S$-restricted Fubini rankings of type~1}.\\
$\UPF$ Type 2:&
Let $\beta \in \UPF_n$  with block structure $\beta' = \pi_1 | \pi_2 | \cdots | \pi_k$ (see \Cref{def:UPF block structure}).
We define $\UPFTypeTwo_{n,S}$ as the set of unit interval parking functions of length $n$ such that the size of each block of $\beta$ belongs to $S$. 
The set $\UPFTypeTwo_{n,S}$, is called the set of \emph{$S$-restricted unit interval parking functions of type~2}.\\

$\FR$ Type 2:&
We define $\FRTypeTwo_{n,S}$ as the set of Fubini rankings with $n$ competitors if, for each entry $x$ in $\alpha\in \FRTypeTwo_{n,S}$, the number of occurrences of $x$ in $\alpha$ belongs to $S$. 
The set $\FRTypeTwo_{n,S}$ is called the set of \emph{$S$-restricted Fubini rankings of type~2}.\\

$\UPF$ Type 3:&
Let $\beta \in \UPF_n$  with block structure $\beta' = \pi_1 | \pi_2 | \cdots | \pi_k$  (see \Cref{def:UPF block structure}).
A sequence $S=(s_1,s_2,\dots)$ of positive integers is called a \textit{block sequence} for $\beta$ if $|\pi_i| \leq s_i$ for each $i=1,\ldots,k$. 
We define $\UPFTypeThree_{n,S}$ as the set of unit interval parking functions of length $n$ with block sequence $S$. 
The set $\UPFTypeThree_{n,S}$ is called the set of \emph{$S$-restricted unit interval parking functions of type~3}.\\
$\FR$ Type 3:&
Let $ \alpha \in \FR_n $  have $k$ distinct ranks. The \emph{position vector} of $\alpha$ is the vector $\mathbf{c}=(c_1,\ldots, c_k)$, where $c_i$ is the number of entries of $\alpha$ that take the value $c_0+c_1+\cdots+c_{i-1}$  for each $i=1,\dots,k$, with $c_0=1$. 

Let $ \alpha \in \FR_n $  have position vector $\bm{c}=(c_1,\dots,c_k)$. A sequence of positive integers  $S=(s_1,s_2,\dots)$ is called a \textit{block sequence} for $\alpha$ if $c_i \leq s_i$ for each $i=1,2, \ldots,k$. We denote by $\FRTypeThree_{n,S}$ the set of Fubini rankings of length $n$ with block sequence $S$. If $\alpha \in \FRTypeThree_{n,S}$, we say that $\alpha$ is \emph{$S$-restricted Fubini rankings of type 3}.
\end{longtable}
\end{center}
By definition, we consider the empty set as a Fubini ranking of size zero, and similarly for unit parking functions of any type. 
\end{definition} 
Our main results include the following.
\begin{itemize}
\item \Cref{biyection}:
For all $n \in \Z^+$ and any set $S$ of positive integers, the set $\FRTypeOne_{n,S}$ is in bijection with $ \UPFTypeOne_{n,S}$.
\item \Cref{thm:egf UPF type 1}: If $S\subseteq\mathbb{Z}^+$, then
$\sum_{n\geq 0} |\UPFTypeOne_{n,S}| \frac{x^n}{n!}=\sum_{k \in S} (e^x -1)^k.$
\item \Cref{thm:UPF type 1 set of evens}:
If $S\subseteq 2\mathbb{Z}^+$, then
$
|\UPFTypeOne_{n,S}| = \sum_{m\in S} \sum_{j=1}^m (-1)^{m-j} \binom{m}{j} j^n$.

\item \Cref{thm:cic1}:  Let $\mathscr{L}_n$ be the set of  ordered set partitions of  $[n]$ with an even number of blocks and let  $\mathscr{O}_n$ be the set of all ordered set partitions of the set $[n]$ with no cyclical adjacencies (\Cref{def:no cyclical adjacencies}).
 Then, for all $n\geq 1$, the
set $\mathscr{L}_n$ is in bijection with $\mathscr{O}_n$.
\item  \Cref{thm:equinumerous}:
For all $n\geq 1$, we have $|\UPFTypeOne_{n, 2 \Z^+}|=|\mathscr{O}^{\UPF}_n|$, 
where $\mathscr{O}^{\UPF}_n$ denotes 
the set of all unit interval parking functions of length $n$ with no cyclical adjacencies (\Cref{def:UPF no cyclical adj}).
\item \Cref{thm:exceedances}:
For $n\geq 2$, 
the elements of $\UPFTypeOne_{n,\{n-1\}}$ are in bijection with the exceedances of the permutations of $[n]$.
In particular, $|\UPFTypeOne_{n,\{n-1\}}|=(n-1)\frac{n!}{2}$, the number of exceedances of the permutations of $[n]$.
\item \Cref{thm:FR type 2 k divides n}:
If $S=\{k\}$ for some $k$ that divides $n$, then $|\FRTypeTwo_{n,S}|=\frac{n!}{(k!)^{n/k}}$.
\item \Cref{thm:FR type 2}:
Let  $S=\{s_1,\dots,s_j\}$ be a finite subset of positive integers, and let $C=\{\bm{c}=(c_1, \ldots, c_j) : \sum_{i=1}^j s_i c_i = n\}$. Then
$$\FRTypeTwo_{n,S}=n! \sum_{\bm{c} \in C} \binom{\sum_{k=1}^j c_{k}}{c_1,\dots,c_j} \prod_{i=1}^j (s_i)!^{-c_i}.$$
\item \Cref{thm:FR type 2 bijection UPF type 2}:
For all $n \in \Z^+$ and any set $S$ of positive integers, 
the set $\FRTypeTwo_{n,S}$ is in bijection with $ \UPFTypeTwo_{n,S}$.
\item \Cref{counting:type3}:
For all $n \in \Z^+$ and every sequence of positive integers $S=(s_1,s_2,\ldots)$, we have 
\begin{align*}
|\FRTypeThree_{n,S} |= \sum_{k=1}^{n} \sum_{(c_1,c_2,\ldots,c_k) \in M_{S,k,n}} \binom{n}{c_1,c_2,\ldots ,c_k}, 
\end{align*}
where 
$M_{S,k,n} = \left\{(c_1, c_2, \dots, c_k)  \mid  c_1+c_2+\cdots + c_k= n  \text { and } c_i\in[s_i] \text{ for all } 1\leq i \leq k\right\}$.
\item \Cref{thm:UPF type 3 bijection FR type 3}
For all $n \in \Z^+$ and any sequence $S$ of positive integers, the set $\FRTypeThree_{n,S}$ is in bijection with $\UPFTypeThree_{n,S}$.
\end{itemize}

This article is organized as follows. 
In \Cref{sec:Preliminar}, we give some preliminary results and recall some main facts about unit interval parking functions, unit interval parking functions, and their connections to set partitions. 
In \Cref{sec:restricted UPF Type 1}, \ref{sec:restricted type 2}, and \ref{sec:restricted type 3} we prove results related to restricted Fubini rankings and restricted unit interval parking functions of type 1, 2, and 3, respectively.
We conclude with \Cref{sec:future} some directions for future study.

\section{Preliminaries}\label{sec:Preliminar}

In this section, we provide definitions, set notation, and establish some preliminary results which help in subsequent sections. 
We begin with the definition of a unit interval parking function.
\begin{definition}\label{def:UPF}
Let $\alpha = (a_1,\dots,a_n)\in[n]^n$ be a parking preference of $n$ cars on a one-way street with $n$ parking spots. We say that $\alpha$ is a \textit{unit interval parking function} of length $n$ if each car $i$ can park either in its preferred spot $a_i$ or in the next one, $a_i+1$. The set of all unit interval parking functions with length $n$ is denoted by $\UPF_n$.
\end{definition}

We illustrate the definition of unit interval parking functions next.

\begin{example}\label{ex1}
Consider $\alpha=(6,1,5,1,2,3)$. This is a unit interval parking function of length $6$, and the final configuration of the cars parking on the street, which is called the outcome of $\alpha$,  is illustrated in \Cref{fig:example of pf}. In contrast,  $\alpha=(1,5,1,3,6,2)$ is not a unit interval parking function, since car $6$ cannot park in either spot $2$  or in spot $3$, see illustration in \Cref{fig:ex2}. This example also illustrates that not every rearrangement of the entries of a unit interval parking function remains a unit interval parking function. 

\begin{figure}[ht!]
    \centering
    \resizebox{\textwidth}{!}{
    \begin{tikzpicture}
        \node at(0,0){\includegraphics[width=1in]{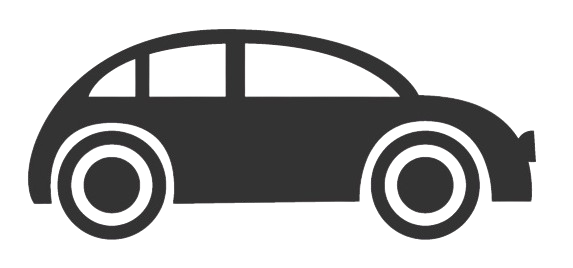}};
        \node at(3,0){\includegraphics[width=1in]{car.png}};
        \node at(6,0){\includegraphics[width=1in]{car.png}};
        \node at(9,0){\includegraphics[width=1in]{car.png}};
        \node at(12,0){\includegraphics[width=1in]{car.png}};
        \node at(15,0){\includegraphics[width=1in]{car.png}};
        \node at(0,-.05){\textcolor{white}{\textbf{2}}};
        \node at(3,-.05){\textcolor{white}{\textbf{4}}};
        \node at(6,-.05){\textcolor{white}{\textbf{5}}};
        \node at(9,-.05){\textcolor{white}{\textbf{6}}};
        \node at(12,-.05){\textcolor{white}{\textbf{3}}};
        \node at(15,-.05){\textcolor{white}{\textbf{1}}};
        \node at(.9,1.){\includegraphics[width=.75in]{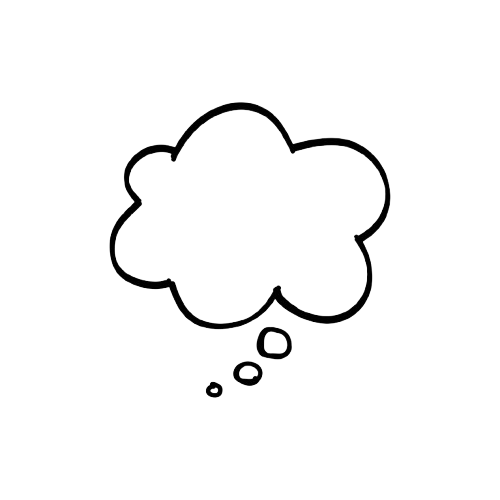}};
        \node at(3.9,1.){\includegraphics[width=.75in]{callout.png}};
        \node at(6.9,1.){\includegraphics[width=.75in]{callout.png}};
        \node at(9.9,1.){\includegraphics[width=.75in]{callout.png}};
        \node at(12.9,1.){\includegraphics[width=.75in]{callout.png}};
        \node at(15.9,1.){\includegraphics[width=.75in]{callout.png}};
        \node at(.9,1.1){\textbf{1}};
        \node at(3.9,1.1){\textbf{1}};
        \node at(6.9,1.1){\textbf{2}};
        \node at(9.9,1.1){\textbf{3}};
        \node at(12.9,1.1){\textbf{5}};
        \node at(15.9,1.1){\textbf{6}};
        \draw[ultra thick](-1.25,-.6)--(1.25,-.6);
        \draw[ultra thick](1.75,-.6)--(4.25,-.6);
        \draw[ultra thick](4.75,-.6)--(7.25,-.6);
        \draw[ultra thick](7.75,-.6)--(10.25,-.6);
        \draw[ultra thick](10.75,-.6)--(13.25,-.6);
        \draw[ultra thick](13.75,-.6)--(16.25,-.6);
        \node at(0,-1){\textbf{1}};
        \node at(3,-1){\textbf{2}};
        \node at(6,-1){\textbf{3}};
        \node at(9,-1){\textbf{4}};
        \node at(12,-1){\textbf{5}};
        \node at(15,-1){\textbf{6}};
        
    \end{tikzpicture}
    }
    \vspace{-.3in}
    \caption{The parking outcome for the unit interval parking function $(6,1,5,1,2,3)$.}
    \label{fig:example of pf}
\end{figure}

\begin{figure}[ht!]
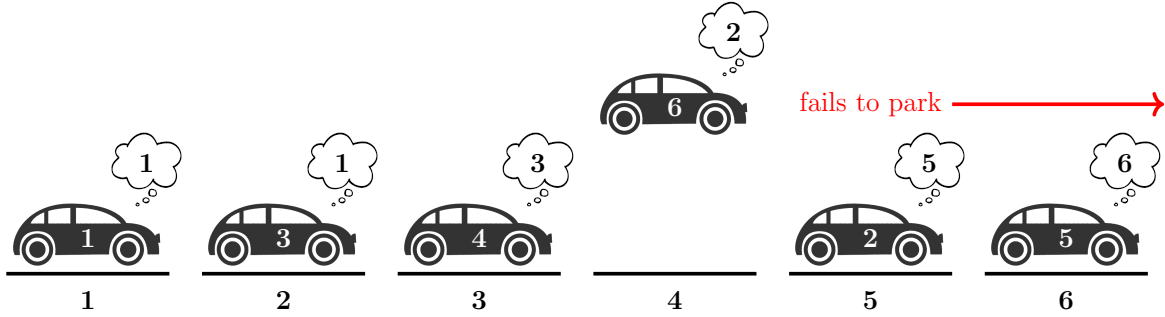

    \centering
    \resizebox{\textwidth}{!}{
    \begin{tikzpicture}
        \node at(0,0){\includegraphics[width=1in]{car.png}};
        \node at(3,0){\includegraphics[width=1in]{car.png}};
        \node at(6,0){\includegraphics[width=1in]{car.png}};
        \node at(9,2){\includegraphics[width=1in]{car.png}};
        \node at(12,0){\includegraphics[width=1in]{car.png}};
        \node at(15,0){\includegraphics[width=1in]{car.png}};
        \node at(0,-.05){\textcolor{white}{\textbf{1}}};
        \node at(3,-.05){\textcolor{white}{\textbf{3}}};
        \node at(6,-.05){\textcolor{white}{\textbf{4}}};
        \node at(9,1.95){\textcolor{white}{\textbf{6}}};
        \node at(12,-.05){\textcolor{white}{\textbf{2}}};
        \node at(15,-.05){\textcolor{white}{\textbf{5}}};
        \node at(.9,1.){\includegraphics[width=.75in]{callout.png}};
        \node at(3.9,1.){\includegraphics[width=.75in]{callout.png}};
        \node at(6.9,1.){\includegraphics[width=.75in]{callout.png}};
        \node at(9.9,3.){\includegraphics[width=.75in]{callout.png}};
        \node at (12,2){\tr{fails to park}};
         \draw[->,ultra thick, red](13.25,2)--(16.5,2);
        \node at(12.9,1.){\includegraphics[width=.75in]{callout.png}};
        \node at(15.9,1.){\includegraphics[width=.75in]{callout.png}};
        \node at(.9,1.1){\textbf{1}};
        \node at(3.9,1.1){\textbf{1}};
        \node at(6.9,1.1){\textbf{3}};
        \node at(9.9,3.1){\textbf{2}};
        \node at(12.9,1.1){\textbf{5}};
        \node at(15.9,1.1){\textbf{6}};
        \draw[ultra thick](-1.25,-.6)--(1.25,-.6);
        \draw[ultra thick](1.75,-.6)--(4.25,-.6);
        \draw[ultra thick](4.75,-.6)--(7.25,-.6);
        \draw[ultra thick](7.75,-.6)--(10.25,-.6);
        \draw[ultra thick](10.75,-.6)--(13.25,-.6);
        \draw[ultra thick](13.75,-.6)--(16.25,-.6);
        \node at(0,-1){\textbf{1}};
        \node at(3,-1){\textbf{2}};
        \node at(6,-1){\textbf{3}};
        \node at(9,-1){\textbf{4}};
        \node at(12,-1){\textbf{5}};
        \node at(15,-1){\textbf{6}};
    \end{tikzpicture}
    }
    \vspace{-.3in}
    \caption{The preference list $(1,5,1,3,6,2)$ is not a unit interval parking function as car $6$ is unable to park at most one spot away from its preference.}
    \label{fig:ex2}
\end{figure}
\end{example}

Bradt et al.\ \cite{UPF} established a connection between unit interval parking functions and Fubini rankings. In this paper, we explore this relationship in more detail. 
Let us begin by giving the technical definition of a Fubini ranking.

\begin{definition}\label{FubRan}
Let $\beta = (b_1,\dots,b_n)\in [n]^n$. We say that $\beta$ is a \textit{Fubini ranking} if it satisfies the following conditions:
\begin{enumerate}
    \item There exists at least one index $i\in [n]$ such that $b_i=1$.
    \item For every $x\in [n]$, if there are $k>0$ entries of $\beta$ equal to $x$, then the next possible value on $\{b_1,\dots,b_n\}$ that is greater than $x$ must be $x+k$.
\end{enumerate}
The set of all Fubini rankings of length $n$ is denoted by $\FR_n$.
The distinct values appearing in a Fubini ranking are called \textit{ranks}.
\end{definition}

We illustrate the definition of Fubini rankings next.
\begin{example}Consider $\beta=(1,4,2,2,6,4)$. It satisfies condition (1) of \Cref{FubRan} since $1$ appears in $\beta$. For condition (2), we verify that each value follows the required pattern: there is one occurrence of $1$, so the next value must be $2 = 1+1$; two occurrences of $2$, so the next must be $4 = 2+2$ (as $3$ is absent); and two occurrences of $4$, so the next must be $6 = 4+2$ (as $5$ is absent). If $6$ were replaced by $5$, condition (2) would fail, and $\beta$ would not be a Fubini ranking.
\end{example}

Note that every Fubini ranking is a parking function, for a proof see~\cite{Hadaway_unit_interval}. Next we define ordered set partitions. 

\begin{definition}\label{def:set partition}
An \textit{ordered  partition} of $[n]$ into $k$ parts (also called blocks) is a sequence $(B_1,B_2,\dots,B_k)$ of disjoint nonempty subsets of $[n]$ whose union is $[n]$. 
We let $\B_n$ denote the set of all ordered  partitions of $n$.
\end{definition}

There is a natural bijection between ordered set partitions of $[n]$ and Fubini rankings with $n$ competitors. 

To construct a Fubini ranking from an ordered set partition $B=(B_1,B_2,\dots,B_k)\in \B_n$, we associate to it a word $(b_1, b_2,\dots, b_n)$ defined by setting   $b_i=1+\sum_{\ell=1}^{j-1}|B_\ell|$ if and only if $i\in B_j$. 
For example, the ordered set partition $(\{2, 5\}, \{1,3,4\}, \{7, 8\}, \{6\})$ corresponds to the ranking $(3,1,3,3,1,8,6,6)$. This construction defines a bijection between ordered set partitions and Fubini rankings.

The \emph{Fubini numbers} (or \emph{ordered Bell numbers}) count the number of \emph{ordered set partitions} of a set with $n$ elements, that is $\Fub_n=|\B_n|$ for all $n\geq 0$. The \emph{Stirling numbers of the second kind}, denoted $\sts{n}{k}$, count the number of ways to partition the set $[n]$ into $k$ nonempty subsets, for more details about set partitions see~\cite{Mezo2}. 
The Fubini numbers are given by (cf.\ \cite{Mezo, Komatsu})
$$
\Fub_n=\sum_{k=0}^n k!\sts{n}{k}.
$$
They can also be expressed using binomial coefficients or as an infinite series:
$$
\Fub_n=\sum_{k=0}^n\sum_{j=0}^k(-1)^{k-j}\binom{k}{j}j^n=\frac{1}{2}\sum_{m=0}^\infty\frac{m^n}{2^m}.
$$
The first few Fubini numbers (sequence \seqnum{A000670} on the OEIS~\cite{OEIS}) are 
$$
1, \quad  1, \quad  3, \quad  13, \quad  75, \quad  541, \quad  4683, \quad  47293, \quad  545835,  \quad  7087261,\dots.$$

We recall that the exponential generating function for the  Fubini numbers is given by
\begin{equation*}
\frac{1}{2-e^x}=\sum_{n=0}^\infty \Fub_n\frac{x^n}{n!}.
\end{equation*}
Moreover, Fubini numbers satisfy the recurrence relation:
\begin{equation*}
\Fub_n=\sum_{j=1}^n\binom{n}{j}\Fub_{n-j}.
\end{equation*}

We now provide a definition first given by Bradt et al.\ \cite{UPF}, which helps with the characterization of unit interval parking functions and gives a connection to set partitions. 

\begin{definition}[Definition 2.7 in~\cite{UPF}]\label{def:UPF block structure}
Let $\alpha=(a_1,\dots,a_n)\in \PF_n$, and let $\alpha^{\uparrow}=(a_1',\dots,a_n')$ be the weakly increasing rearrangement of $\alpha$. The \textit{block structure} of $\alpha$, denoted by $\alpha'=\pi_1|\pi_2|\cdots|\pi_k$, is the partition of $\alpha^{\uparrow}$ into consecutive block $\pi_j$, where each block $\pi_j$ starts at the $j$-th entry $a_i'$ of $\alpha^{\uparrow}$  satisfying $a_i'=i$. The elements  $\pi_1,\dots,\pi_k$ are called \textit{blocks}. 
\end{definition}

We illustrate the block structure of a unit interval parking function next.
\begin{example}
Consider $\alpha=(7,5,1,1,5,7,2,8,3)\in \UPF_9$. Its weakly increasing rearrangement is  $\alpha^{\uparrow}=(1,1,2,3,5,5,7,7,8)$. Identifying the positions where $a_i' = i$, we obtain the block structure  $\alpha'=1123|55|778$, with blocks $\pi_1=1123$, $\pi_2=55$, and $\pi_3=778$.
\end{example}

As we showed in \Cref{ex1}, a rearrangement of a unit interval parking functions is not necessarily a unit interval parking function. However, certain rearrangements do preserve this property.  
The following result, established by Bradt et al.\ \cite{UPF}, characterizes these cases.

\begin{theorem}[Theorem 2.9 in~\cite{UPF}]\label{T1}
Let $\alpha=(a_1,\dots,a_n)\in\UPF_n$ and $\alpha'=\pi_1|\cdots|\pi_k$ be the block structure of $\alpha$. Then
\begin{enumerate}
    \item A rearrangement $\sigma$ of $\alpha$ belongs to $\UPF_n$ if and only if it preserves the relative  order of the entries within each block of $\alpha$.

    \item The number of such  rearrangements in $\UPF_n$ is given by the multinomial coefficient
    \[
    \binom{n}{|\pi_1|,\dots,|\pi_k|}.
    \]
\end{enumerate}
\end{theorem}

\begin{remark} \label{R:ind.form.blocks}
From the proof of the previous theorem, we highlight two key observations: 
\begin{enumerate}
    \item Each block in the block structure of a unit interval parking function $\alpha$ is independent. That is, a car $c_i$ in $\alpha$ finds its parking spot occupied only by other cars within the same block of $c_i$. 
    \item The general form of every block $\pi_i$ in the block structure of some $\alpha\in \UPF_n$ is $\pi_i= k k (k+1)\cdots (k+\ell-2),$ where $k=\min \pi_i$ and $\ell=|\pi_i|$. 
    \end{enumerate}
\end{remark}

Bradt et al.\ \cite[Theorem 2.5]{UPF}  provide a bijection between $\FR_n$ and $\UPF_n$ for $n\geq 1$ using two special functions that relate the block structure of a unit interval parking function to the different ranks of a Fubini ranking. These functions are defined as follows, see~\cite[Lemma 2.6]{UPF}. The function  
\begin{align}
\phi: \FR_n \to \UPF_n\label{phi}
\end{align}
is defined by  $\phi(\beta)=(a_1,\dots,a_n)$, where $\beta=(b_1,\dots,b_n)\in \FR_n$ and
\[
a_i =
\begin{cases}
    b_i, & \text{ if } \hspace{0.3cm} |\{ j\in[i] : b_j=b_i \}|=1;\\
    b_i+k-2, & \text{ if } \hspace{0.3cm} |\{ j\in[i] : b_j=b_i \}|=k>1.
\end{cases}
\]
Similarly, the function \begin{align}
\psi: \UPF_n\to \FR_n\label{psi}
\end{align}is defined by $\psi(\alpha)=(b_1,\dots,b_n)$, for each $\alpha=(a_1,\dots,a_n) \in \UPF_n$, where $b_i= \min \{ a\in \pi_j: a_i\in \pi_j \}$ and $\pi_j$ is the $j$-th block of the block structure of $\alpha$.

Let us give an example of how these functions work.
\begin{example}
Consider  $\alpha=(2,2,5,3,5,6,1,7,8)\in \UPF_9$. The block structure of $\alpha$ is $\alpha'= 1|223|5567=\pi_1|\pi_2|\pi_3$. Applying $\psi$ to $\alpha$, we obtain
\[
\psi((2,2,5,3,5,6,1,7,8))=(2,2,5,2,5,5,1,5,5).
\]
Finally, applying $\phi$ to $\psi(\alpha)$ recovers $\alpha$, as expected.
\end{example}

It is important to note that there exists a correspondence between the blocks in the ordered set partitions and the ranks in the Fubini rankings. The number of blocks is equal to the number of distinct ranks, and the number of elements in a given block corresponds to the number of ties at the associated rank in the Fubini ranking.

Combining \Cref{T1} with the fact that $|\UPF_n|=\Fub_n=|\B_n|$, Bradt et al.\ obtain the following identity for the Fubini numbers, see~\cite[Corollary 2.13]{UPF}: For  $n\geq 1$, let $(c_1,\dots,c_k)\vDash n$ denote a composition of $n$ into  $k$ parts. Then, as described in 
\[
\Fub_n = \sum_{k=1}^n \sum_{(c_1,\dots,c_k)\vDash n} \binom{n}{c_1,\dots,c_k}.
\]

With these preliminary results at hand, we are now ready to introduce restricted unit interval parking functions and restricted Fubini rankings. 
\section{Restricted Fubini rankings and Restricted Unit Interval Parking Functions of Type 1}\label{sec:restricted UPF Type 1}

In this section, we introduce the notion of \emph{lucky cars} and describe their role in unit interval parking functions. 
This concept is based on the block structure of elements in $\UPF_n$ as was  described by Bradt et al.\ in~\cite{UPF}. 

\begin{definition}\label{def:lucky car}
Let $\alpha=(a_1,\dots,a_n)\in \PF_n$. We say that car $i$ is a \textit{lucky car} if it parks in its preferred spot $a_i$.
\end{definition}

For example, in the parking function $\alpha=(5,2,1,1,3,5,4)$, cars 1, 2, and 3 are lucky. 
Whenever the parking function is also a unit interval parking function, 
the number of blocks in its block structure is the same as the number of lucky cars. 
We prove this next.

\begin{theorem}\label{Luckyblocks}
Let $\alpha \in \UPF_n$. Then, the number of lucky cars is equal to  the  number of blocks in the block structure of $\alpha$.
\end{theorem}

\begin{proof}
Let $\alpha=(a_1,\dots,a_n)\in \UPF_n$, and let  $\alpha^{\uparrow}=(a_1',\dots,a_n')$ be its weakly increasing rearrangement. 
Denote the block structure of $\alpha$  by  $\alpha'=\pi_1| \cdots |\pi_m$.  
From \Cref{T1}, we know that the relative order of the elements within each block $\pi_i$ is preserved in $\alpha$. 
Moreover, by \Cref{R:ind.form.blocks}, the blocks of $\alpha$ are independent, and each block $\pi_i$ has the form $\pi_i=(k,k,k+1,\dots,k+\ell-2)$, where $\ell=|\pi_i|$ is the number of elements in $\pi_i$, and $k=\min \pi_i$.  
Since the first car in each block can always park in its preferred spot,  $\alpha$ has at least $m$ lucky cars. 

On the other hand, suppose that cars $i_1,\dots,i_\ell$ are the $\ell$ cars with preferences in $\pi_i$, where $i_1 < \cdots <i_\ell$.  
Then, their parking preferences are $a_{i_1}=k$, $a_{i_2}=k,\dots,a_{i_\ell}=k+\ell-2$. 
Hence, the parking process proceeds as follows: car ${i_1}$ parks in spot $k$. Car ${i_2}$ attempts to park in spot $k$, but since it is occupied, it parks in spot $k+1$. Car ${i_3}$ parks at the next available  spot, which is numbered $k+2$,  and so on. 
Thus,  the only lucky car in $\pi_i$ is car ${i_1}$. 
Since each block $\pi_i$ contributes exactly one lucky car, it follows that the number of lucky cars in $\alpha$ is exactly the number of blocks, which is $m$, as claimed.
\end{proof}

We now introduce our main object of study in this section, which is unit interval parking functions with a number of lucky cars in a given set of integers.

\begin{definition}\label{def:restricted UPF  type 1}
Let $S$ be a subset of positive integers. We define $\UPFTypeOne_{n,S}$  as the set of unit interval parking functions of length $n$ such that the number of lucky cars belongs to $S$. If $\alpha \in \UPFTypeOne_{n,S}$, we say that $\alpha$ is \emph{$S$-restricted unit interval parking functions of type 1}. 
\end{definition}

For example, consider $\alpha=(1,4,4,1,5,2,7,6)\in \UPFTypeOne_n$. The block structure of $\alpha$ is $\alpha'=112|4456|7$, which, by \Cref{Luckyblocks}, implies that $\alpha$ has exactly three lucky cars. Hence, $\alpha\in\UPFTypeOne_{8,\{ 1,2,3 \}}$. By contrast, let  $\beta=(1,3,4,7,4,6,2,7)$. Since $\beta$ has six lucky cars, we conclude that $\beta \notin\UPFTypeOne_{8,\{ 1,2,3 \}}$.

\begin{definition}\label{def:restricted FR type 1} We define $\FRTypeOne_{n,S}$ as the set of Fubini rankings with $n$ competitors such that the number of distinct ranks belongs to $S$.
The set $\FRTypeOne_{n,S}$, is called the set of \emph{$S$-restricted Fubini rankings of type 1}.
\end{definition}

For example, consider $\gamma=(1,1,1,4,5,6,7,7,9)\in\FR_9$, which has 6 ranks, and so $\gamma\in\FRTypeOne_{9,S}$, for any $S\subseteq\mathbb{Z}^+$ provided $6\in S$. However, $\gamma\notin \FRTypeOne_{9,\{1,5\}}$, as $6\notin \{1,5\}$.

In the proof of the following result we use the functions $\phi$ and $\psi$ defined in \eqref{phi} and \eqref{psi}, respectively.
\begin{theorem}\label{biyection}
For all $n \in \Z^+$ and any set $S$ of positive integers, the set $\FRTypeOne_{n,S}$ is in bijection with $ \UPFTypeOne_{n,S}$.
\end{theorem}

\begin{proof}
Given $\alpha \in \UPFTypeOne_{n,S}$, we know from \Cref{Luckyblocks} that the number of blocks in $\alpha'$ (the block structure of $\alpha$) belongs to $S$. According to~\cite[Lemma 2.11]{UPF}, each block in $\alpha'$ is mapped by $\psi$ to a list of identical entries in $\psi(\alpha)$.  These lists determine the  distinct  ranks in the Fubini ranking $\psi(\alpha)$, which implies that  $ \psi(\alpha) \in \FRTypeOne_{n,S}$. 

Conversely, let $\beta \in \FRTypeOne_{n,S}$. Then, there exist $k$ distinct values among the entries of  $\beta$ for some $k \in S$. Each list of all entries with a unique different value of $\beta$ (what we call \emph{competitors in the same rank}) is transformed by $\phi$ into a block of the unit interval parking function $\phi(\beta)$. Since all blocks of  $\phi(\beta)$ are generated in this way, we conclude, by \Cref{Luckyblocks}, that $\phi(\beta) \in \UPFTypeOne_{n,S}$.
\end{proof}

For a fixed subset $S\subseteq\mathbb{Z}^+$, we can give a closed formula for the exponential generating function for $\{|\UPFTypeOne_{n,S}|\}_{n\geq 1}$.

\begin{theorem}\label{thm:egf UPF type 1}
If $S\subseteq\mathbb{Z}^+$, then
$$\sum_{n\geq 0} |\UPFTypeOne_{n,S}| \frac{x^n}{n!}=\sum_{k \in S} (e^x -1)^k.$$ 
\end{theorem}
\begin{proof}
We recall that unit interval parking functions are counted by Fubini numbers, meaning that for all $n\geq 0$,  we have $\Fub_n=|\UPF_n|$. Now, from the bijection given in \Cref{biyection},  it follows that  $|\UPFTypeOne_{n,S}|= \sum_{k \in S}k! \sts{n}{k}$. Since the exponential generating function (e.g.f.) of the sequence $\{k! \left\{ \sts{n}{k} \right\}\}_{n \geq 0}$ is given by $(e^x -1)^k$, we obtain that 
\[\sum_{n\geq 0} |\UPFTypeOne_{n,S}| \frac{x^n}{n!}=\sum_{k \in S} (e^x -1)^k.\qedhere\]
\end{proof}
We now specialize to certain subsets of the integers and give the e.g.f.\ for those cases. Remember, 
that  $|\UPFTypeOne_{0,S}|=|\{\emptyset \}|=1$. 

\begin{corollary}\label{lemma:egf examples} The closed-form e.g.f.\ of the counting sequence $\{ |\UPFTypeOne_{n,S}|\}_{n \geq 0}$ for some special choices of \( S \subseteq\mathbb{Z}^+\) are as follows.
\begin{enumerate}
\item If $S=\{k \in \Z^+ : k$ is even$\}$, then 
\begin{align*}
\sum_{k \in S} (e^x -1)^k = \sum_{k \geq 0} (e^x -1)^{2k}= \frac{1}{1-(e^x-1)^2}.
\end{align*}
The first few terms of this sequence correspond to OEIS sequence   \seqnum{A052841}:
\[1,  0,  2,   6,   38,   270,   2342,   23646,  272918,  3543630,  51123782,  \dots
\] 
\item  If $S=\{k \in \Z^+ : k$ is odd$\}$, then 
\begin{align*}
\sum_{k \in S} (e^x -1)^k = \sum_{k \geq 0} (e^x -1)^{2k+1}= \frac{e^x -1}{1-(e^x-1)^2}.
\end{align*}
The first few terms of this sequence correspond to OEIS sequence \seqnum{A089677}:
\[
1,  1,  7,  37,  271,  2341,  23647,  272917,  3543631,  51123781, \dots
\]
\item  Fix $m\in\mathbb{Z}^+$. If $S=\{k \in \Z^+ : k \geq m\}$ or $S=\{k \in \Z^+ : k \leq m\}$, then the corresponding generating functions are 
\begin{align*}
\sum_{k \in S} (e^x -1)^k &= \sum_{k \geq m} (e^x -1)^{k}= \frac{(e^x -1)^m}{2-e^x}\\
\intertext{and}
\sum_{k \in S} (e^x -1)^k &= \sum_{k = 1}^m (e^x -1)^{k}= \frac{1- (e^x -1)^{m+1}}{2-e^x},\intertext{respectively.}
\end{align*} 
\end{enumerate}
\end{corollary}
\begin{proof}
{All of the statements are proved by using the geometric series $1/(1-x)=\sum_{k\geq 0} x^k$ and the finite geometric sum $(1-x^{m+1})/(1-x)=\sum_{k=0}^m x^k$.}
\end{proof}

\begin{remark}
The sum of the sequences appearing in items (1) and (2) of \Cref{lemma:egf examples} yields the Fubini numbers.
\end{remark}

We can also obtain some closed formulas for these functions.
\begin{corollary}\label{cor:S all the evens}
For all $n\geq 0$, if $S=2\mathbb{Z}^+$, then 
\begin{align*}
|\UPFTypeOne_{n, 2\Z^+}|= \sum_{k=1}^{\floor{n/2}} \sum_{i=0}^{2k} \binom{2k}{i} (-1)^i i^n.
\end{align*}
\end{corollary}
\begin{proof}
To determine $|\UPFTypeOne_{n, 2\Z^+}|$ for all $n \geq 0$, we can first expand $(e^x-1)^{2k}$ using the binomial theorem for a fixed integer $k$ with $2k \leq n$. Doing so yields
\begin{align*}
(e^x-1)^{2k}= \sum_{i=0}^{2k} \binom{2k}{i}e^{ix}(-1)^{2k-i}
&= \sum_{n \geq 2k}\left(\sum_{i=0}^{2k} \binom{2k}{i} (-1)^i i^n\right) \frac{x^n}{n!}.
\end{align*}
Therefore, 
\[
|\UPFTypeOne_{n, 2\Z^+}|= \sum_{k=1}^{\floor{n/2}} \sum_{i=0}^{2k} \binom{2k}{i} (-1)^i i^n.\qedhere
\]
\end{proof}

In a similar manner, one can establish the following result using exponential generating functions. However, we present a combinatorial proof instead.

\begin{theorem}\label{thm:UPF type 1 set of evens}
Let $S\subseteq 2\mathbb{Z}^+$ be a set of positive integers. Then
\[
|\UPFTypeOne_{n,S}| = \sum_{m\in S} \sum_{j=1}^m (-1)^{m-j} \binom{m}{j} j^n.
\]
\end{theorem}

\begin{proof}
First, by the principle of addition, it is enough to prove that the number of unit interval parking functions in $\UPF_n$ with exactly $m$ lucky cars is given by
\[
\sum_{j=1}^m (-1)^{m-j} \binom{m}{j} j^n.
\]
We prove this using the principle of inclusion-exclusion. 

By \Cref{Luckyblocks}, a unit interval parking function $\alpha\in\UPF_n$ has $m$ lucky cars if and only if the number of blocks in $\alpha$'s block structure is $m$. 
Thus, we seek to count the number of unit interval parking functions whose block structure consists of exactly $m$ blocks.
Moreover, by \Cref{T1}, it is not necessary for all  cars to specify their favorite spot, it suffices to determine  which block they prefer, since the cars within each block park in the relative order of their block and the structure of the block (see  \Cref{R:ind.form.blocks}). 

Thus, we count the number of list of block preferences in which every block is chosen by at least one car. In other words,  we are counting the size of the set
\[
S_n = \{ (b_1,\dots,b_n)\in [m]^n : \text{for every $1\leq i\leq m$ there exist $1\leq k\leq n$ such that $b_k=i$}\}.
\]
Equivalently, we have \( |\UPFTypeOne_{n,\{m\}}| = |S_n| \).  

Now, define 
\[
A_i = \{ (b_1,\dots,b_n)\in [m]^n : \text{there is no $1\leq k\leq n$ such that $b_k=i$}\}
\]
for each $1\leq i\leq m$. Since $S_n = [m]^n - \bigcup_{n=1}^m A_i$, and, it follows that \begin{align}
|\UPFTypeOne_{n,\{m\}}| = m^n - \left|\bigcup_{n=1}^m A_i \right|\label{previous expression}.
\end{align}
Applying the principle of inclusion-exclusion, we obtain
\[
\left|\bigcup_{i=0}^m A_i\right| = \sum_{j=1}^m (-1)^{j+1} 
\sum_{1\leq i_1< \cdots <i_j\leq m} \left| A_{i_1} \cap \cdots \cap A_{i_j} \right|.
\]
To compute  $|A_{i_1} \cap \cdots \cap A_{i_j}|$, suppose $1\leq i_1< \cdots <i_j\leq m$ and observe that the elements of this intersection are $n$-tuples  $(a_1,\dots,a_n)\in [m]^n$ such that  $a_k\notin \{i_1,\dots,i_j\}$ for every $1\leq k\leq n$.  This implies $A_{i_1} \cap \cdots \cap A_{i_j}=([m]-\{i_1,\dots,i_j\})^n$, so its size is given by  $$|A_{i_1} \cap \cdots \cap A_{i_j}|=(m-j)^n.$$ 
Thus, we obtain
\begin{align}
\left|\bigcup_{i=0}^m A_i\right| &= \sum_{j=1}^m (-1)^{j+1} \sum_{1\leq i_1< \cdots <i_j\leq m} |A_{i_1} \cap \cdots \cap A_{i_j}|\nonumber\\
&= \sum_{j=1}^m (-1)^{j+1} \sum_{1\leq i_1< \cdots <i_j\leq m} (m-j)^n\nonumber \\
&= \sum_{j=1}^m (-1)^{j+1} \binom{m}{j} (m-j)^n.\label{sub this}
\end{align}

Finally, substituting \eqref{sub this} into \eqref{previous expression}, the previous expression for \( |\UPFTypeOne_{n,\{m\}}| \), we have
\begin{align*}
|\UPFTypeOne_{n,\{m\}}| &=
m^n-\sum_{j=1}^m (-1)^{j+1} \binom{m}{j} (m-j)^n=\sum_{j=0}^m (-1)^{j} \binom{m}{j} (m-j)^n.
\end{align*}
Using the symmetry of the binomial coefficient
$\binom {m}{j}=\binom {m}{m-j}$ for $ 0\leq j\leq m$,	
we  rearrange the sum to obtain 
\begin{align*}
|\UPFTypeOne_{n,\{m\}}| = \sum_{j=0}^m (-1)^{m-j} \binom{m}{j} j^n,
\end{align*}
which is the desired formula.
\end{proof}

\subsection{UPFs without cyclical adjacencies}

We consider a class of unit interval parking functions closely related to  $\UPFTypeOne_{n, 2\Z^+}$. Since unit interval parking functions of length $n$ correspond to ordered set partitions of $[n]$ (they are both counted by Fubini numbers), the set $\UPFTypeOne_{n, 2\mathbb{Z}^+}$ is in bijection with ordered set partitions of \([n]\) with an even number of blocks, recalling \Cref{Luckyblocks}. This motivates the introduction of a new class of ordered set partitions, see comments in the sequence \seqnum{A052841} in the OEIS.

\begin{definition}\label{def:no cyclical adjacencies}
Let $\bm{C}=(C_1,\ldots,C_k)$ be an ordered set partition of $[n]$. We say that $\bm{C}$ has \textit{no cyclical adjacencies} if, for every $i=1,\dots,k$ and every $j=1,\dots,n-1$, the following conditions hold:
\begin{enumerate}
 \item $j$ and $j+1$ do not belong to the same block $C_i$, and 
 \item $1$ and $n$ do not belong to the same block $C_i$.
\end{enumerate}
The set of all ordered set partitions of the set $[n]$ with no cyclical adjacencies is denoted by $\mathscr{O}_n$.
\end{definition}

For example, $(\{ 2, 4, 6 \}, \{ 3,5\}, \{1 \})$ and $(\{ 4, 6 \}, \{ 3,5\}, \{1,3 \})$ are elements of $\mathscr{O}_6$, while $(\{1,2,3\},\{4\},\{5,6\})\notin\mathscr{O}_6$.

Let $\mathscr{L}_n$ be the set of  ordered set partitions of  $[n]$ with an even number of blocks. Given any $\mathbf{B}=(B_1,B_2,\ldots,B_{2k}) \in \mathscr{L}_n$, we define  the relation $\prec$ on $[n]$ as follows:
\begin{itemize}
\item  $a \prec b$ if $a <n$ and $b=a+1$,
\item $n \prec 1$.
\end{itemize}
Whenever $a\prec b$, we refer to $a$ as the \textit{cyclical predecessor} of $b$. 

For the next theorem, we need the following function $\lambda_n: \mathscr{L}_n \rightarrow \mathscr{O}_n$  via the following algorithm:
\begin{enumerate}
\item Initialize $i=1$ and  set $B_{2k+1}=\emptyset$.
\item If $i\leq 2k$, proceed to step (3), otherwise, go to step (4).
\item For each chain of elements in $B_i$ of the form $j_1 j_2 \ldots j_t$  (with $t < n$),  where $j_{s} \prec j_{s+1}$ and $(j_s, j_{s+1})$ forms a cyclical adjacency for $s=1,\dots,t-1$, remove the elements $\{j_2, j_4, \ldots, j_{\gamma(t)}\}$ from $B_i$ and place them  in $B_{2k+1}$, where \[\gamma(t)=
\begin{cases}
    t, & \text{if } t \text{ is even;} \\
    t - 1, & \text{if } t \text{ is odd.}
\end{cases} \]
Increment $i$ and return to step (2).
 \item Finally, define $\lambda_n(\mathbf{B})=(B_1,B_2,\ldots,B_{2k},B_{2k+1})$.
\end{enumerate}

Clearly, $\lambda_n(\mathbf{B})$ is a partition of $[n]$. Moreover, every block in $\lambda_n(\mathbf{B})$ avoids cyclical adjacencies because we remove the adjacencies of the original blocks $B_1, B_2,\ldots,B_{2k}$ of $\mathbf{B}$ in the previous algorithm, and the block $B_{2k+1}$ consists of elements from different chains  (which are thus nonadjacent) or from chains where adjacency has been explicitly removed.  Hence, $\lambda_n(\mathbf{B}) \in \mathscr{O}_n$ proving that $\lambda_n$ is well defined.

Let us give an example of how this function works.

\begin{example}
We consider all elements of $\mathscr{L}_4$ that contain cyclical adjacencies and provide the image under $\lambda_4$:
\begin{align*}
&\lambda_4(\{1,2,3\}, \{ 4\})= (\{1,3\},\{ 4\},\{ 2\}),  \quad \lambda_4( \{ 4\}, \{1,2,3\})= (\{ 4\},\{1,3\},\{ 2\}),\\
&\lambda_4(\{1\}, \{2,3,4\}) = (\{1\}, \{2,4\}, \{3\}), \quad 
  \lambda_4(\{2,3,4\}, \{1\}) = (\{2,4\}, \{1\}, \{3\}),\\
&\lambda_4(\{1,2\}, \{3,4\}) = (\{1\}, \{3\}, \{2,4\}), \quad 
\lambda_4(\{3,4\}, \{1,2\}) = (\{3\}, \{1\}, \{2,4\}),\\
&\lambda_4(\{1,3,4\}, \{2\}) = (\{1,3\}, \{2\}, \{4\}), \quad \lambda_4(\{2\}, \{1,3,4\}) = (\{2\}, \{1,3\}, \{4\}),\\
&\lambda_4(\{1,4\}, \{2,3\}) = (\{4\}, \{2\}, \{1,3\}), \quad \lambda_4(\{2,3\}, \{1,4\}) = (\{2\}, \{4\}, \{1,3\}),\\
&\lambda_4(\{1,2,4\}, \{3\}) = (\{2,4\}, \{3\}, \{1\}), \quad 
 \lambda_4(\{3\}, \{1,2,4\}) = (\{3\}, \{2,4\}, \{1\}).
\end{align*}
\end{example}

\begin{theorem}\label{thm:cic1}  For all $n \in \Z^+$, the function $\lambda_n: \mathscr{L}_n \rightarrow \mathscr{O}_n$ is a bijection.
\end{theorem}

\begin{proof}
 \textit{Injectivity:}  Suppose $\mathbf{A} = (A_1, A_2, \dots, A_q)$ and $\mathbf{B} = (B_1, B_2, \dots, B_p)$ are distinct elements of $\mathscr{L}_n$, meaning there exists some $j \in [1, \max(p, q)]$ such that $A_j \neq B_j$.  By following the construction of $\lambda_n(\mathbf{A})$ and $\lambda_n(\mathbf{B})$, at least one of the following must hold: 
\begin{enumerate}
    \item The elements of $A_j$ in $\lambda_n(\mathbf{A})$ differ from those of $B_j$ in $\lambda_n(\mathbf{B})$.
    \item The elements of $A_{q+1}$ in $\lambda_n(\mathbf{A})$ differ from those of $B_{p+1}$ in $\lambda_n(\mathbf{B})$.
\end{enumerate}

Thus, $\lambda_n(\mathbf{A}) \neq \lambda_n(\mathbf{B})$, proving injectivity.

 \textit{Surjectivity:} Given any $\mathbf{C}=(C_1,C_2,\ldots,C_p) \in \mathscr{O}_n$, we construct a preimage $\mathbf{B} \in \mathscr{O}_n$ such that $\lambda_n(\mathbf{B})= \mathbf{C}$. If $\mathbf{C}$ has an even number of blocks, then $\mathbf{C} \in \mathscr{O}_n$, and by the  algorithm, $\lambda_n(\mathbf{C})=\mathbf{C}$. 

If $\mathbf{C}$ has an odd number of blocks, take the last block $C_p = \{ c_1, c_2, \dots, c_m \}$, where $m < n$ (since $\mathbf{C}$ does not contain all elements of $[n]$). Move each $c_l$ to the same block as its cyclical predecessor. The resulting partition has an even number of blocks, and following the algorithm, its image under $\lambda_n$ is precisely $\mathbf{C}$. 

Thus, $\lambda_n$ is bijective.
\end{proof}

\begin{definition}\label{def:UPF no cyclical adj} Let $\alpha=(a_1,\ldots,a_n) \in \UPF_n$ with block structure $\alpha'= \pi_1|\cdots| \pi_k$. We say that $\alpha$ \textit{has no cyclical adjacencies} if, for every $i=1,\ldots,k$ and every $j=1,\dots,n-1$, the following conditions hold:
\begin{itemize}
 \item $a_j$ and $a_{j+1}$ do not belong to  the same block $\pi_i$ simultaneously, and 
 \item $a_1$ and $a_n$ do not belong to same block $\pi_i$ simultaneously. 
\end{itemize}
 The set of all unit interval parking functions of length $n$ with no cyclical adjacencies is denoted by $\mathscr{O}^{\UPF}_n$.
\end{definition}

{An intuitive way to understand unit interval parking functions with no cyclical adjacencies, is to note that adjacent cars in the preference list do not compete to park between them.}

For example, consider $\alpha=(a_1,\ldots,a_7)=(1,4,1,6,2,6,4) \in \UPF_7$, which parks the cars as illustrated in Figure \ref{fig:apart in queue close to park}. The block structure of the unit interval parking function $\alpha=(1,4,1,6,2,6,4)$ is given by $112|44|66= a_1 a_3 a_5 | a_2 a_7 | a_4 a_6$, which satisfies the conditions in \Cref{def:UPF no cyclical adj}. Therefore, we conclude that  $\alpha \in \mathscr{O}^{\UPF}_7$.

\begin{figure}[ht!]
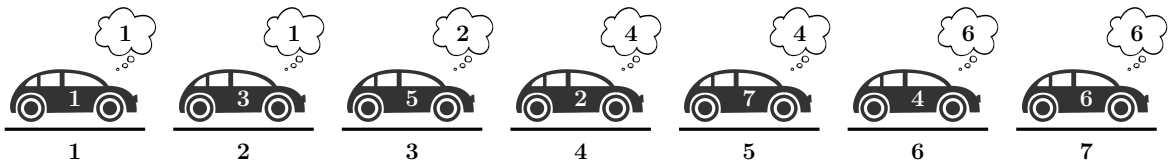

    \centering
    \resizebox{\textwidth}{!}{
    \begin{tikzpicture}
        \node at(0,0){\includegraphics[width=1in]{car.png}};
        \node at(3,0){\includegraphics[width=1in]{car.png}};
        \node at(6,0){\includegraphics[width=1in]{car.png}};
        \node at(9,0){\includegraphics[width=1in]{car.png}};
        \node at(12,0){\includegraphics[width=1in]{car.png}};
        \node at(15,0){\includegraphics[width=1in]{car.png}};
        \node at(18,0){\includegraphics[width=1in]{car.png}};
        \node at(0,-.05){\textcolor{white}{\textbf{1}}};
        \node at(3,-.05){\textcolor{white}{\textbf{3}}};
        \node at(6,-.05){\textcolor{white}{\textbf{5}}};
        \node at(9,-.05){\textcolor{white}{\textbf{2}}};
        \node at(12,-.05){\textcolor{white}{\textbf{7}}};
        \node at(15,-.05){\textcolor{white}{\textbf{4}}};
        \node at(18,-.05){\textcolor{white}{\textbf{6}}};
        \node at(.9,1.){\includegraphics[width=.75in]{callout.png}};
        \node at(3.9,1.){\includegraphics[width=.75in]{callout.png}};
        \node at(6.9,1.){\includegraphics[width=.75in]{callout.png}};
        \node at(9.9,1.){\includegraphics[width=.75in]{callout.png}};
        \node at(12.9,1.){\includegraphics[width=.75in]{callout.png}};
        \node at(15.9,1.){\includegraphics[width=.75in]{callout.png}};
        \node at(18.9,1.){\includegraphics[width=.75in]{callout.png}};
        \node at(.9,1.1){\textbf{1}};
        \node at(3.9,1.1){\textbf{1}};
        \node at(6.9,1.1){\textbf{2}};
        \node at(9.9,1.1){\textbf{4}};
        \node at(12.9,1.1){\textbf{4}};
        \node at(15.9,1.1){\textbf{6}};
        \node at(18.9,1.1){\textbf{6}};
        \draw[ultra thick](-1.25,-.6)--(1.25,-.6);
        \draw[ultra thick](1.75,-.6)--(4.25,-.6);
        \draw[ultra thick](4.75,-.6)--(7.25,-.6);
        \draw[ultra thick](7.75,-.6)--(10.25,-.6);
        \draw[ultra thick](10.75,-.6)--(13.25,-.6);
        \draw[ultra thick](13.75,-.6)--(16.25,-.6);
        \draw[ultra thick](16.75,-.6)--(19.25,-.6);
        \node at(0,-1){\textbf{1}};
        \node at(3,-1){\textbf{2}};
        \node at(6,-1){\textbf{3}};
        \node at(9,-1){\textbf{4}};
        \node at(12,-1){\textbf{5}};
        \node at(15,-1){\textbf{6}};
        \node at(18,-1){\textbf{7}};
    \end{tikzpicture}
    }
    \vspace{-.3in}
    \caption{The parking outcome for the unit interval parking function $(1,4,1,6,2,6,4)$.}
    \label{fig:apart in queue close to park}
\end{figure}

\begin{theorem}\label{thm:equinumerous}
For all $n\geq 1$, we have $|\mathscr{O}^{\UPF}_n|=|\UPFTypeOne_{n, 2 \Z^+}|$.
\end{theorem}

\begin{proof}
By the bijection between ordered set partitions of the set $[n]$ and unit interval parking functions of length $n$, it follows directly that $|\mathscr{O}_n|=|\mathscr{O}^{\UPF}_n|$ and $|\mathscr{L}_n|= |\UPFTypeOne_{n, 2 \Z^+}|$. Thus, by \Cref{thm:cic1}, we conclude the desired result. 
\end{proof}
\Cref{thm:equinumerous}, together with Part (1) of \Cref{lemma:egf examples} and \Cref{cor:S all the evens}, readily imply the following.
\begin{corollary}
The exponential generating function associated with the counting sequence $\{|\mathscr{O}^{\UPF}_n|\}_{n \geq 1}$ is given by 
$$\sum_{n\geq 1}|\mathscr{O}^{\UPF}_n| \frac{x^n}{n!}=\frac{(e^x -1)^2}{1-(e^x-1)^2}.$$
Moreover, for all $n \geq 1$, we have  $$|\mathscr{O}^{\UPF}_n|=\sum_{k=1}^{\floor{n/2}} \sum_{i=0}^{2k} \binom{2k}{i} (-1)^i i^n.$$
\end{corollary}

\subsection{Exceedances of permutations and unit interval parking functions}
Throughout we let $\mathfrak{S}_n$ denote the set of all permutations of $[n]$. We begin with the definition of an exceedance of a permutation. 
\begin{definition}\label{def:exceedance}
Given a permutation $\sigma\in\mathfrak{S}_n$, an \textit{exceedance} in $\sigma$ is an index  $j\in [n]$ such that $\sigma(j)>j$.
\end{definition}

For example, the exceedances in $\sigma=24318576$ are at indices $1, 2$, and $5$.

\begin{remark}
It is possible to  show that the number of exceedances among all permutations of $[n]$ is $(n-1)\frac{n!}{2}$, see comments in the sequence \seqnum{A001286} in the OEIS. 
\end{remark}

Since $\UPFTypeOne_{n,\{n-1\}}$ is in bijection with the set of all ordered set partitions of $[n]$ with $n-1$ blocks, we can define the function $\gamma_n:\UPFTypeOne_{n,\{n-1\}} \rightarrow \mathrm{Exc}_n$, where $\mathrm{Exc}_n$ is the set of exccedances in all permutations of $[n]$ as follows:
Let $\bm{B}=(B_1,\dots,B_i,\dots,B_{n-1})=(\{a_1\},\dots,\{a_i,a_{i+1}\},\dots, \{a_n\})$ be an ordered set partition of $[n]$ with $n-1$ blocks, where $a_i<a_{i+1}$. Define $\gamma_n(\bm{B})=d_1\cdots d_n$, where
\[
d_j=
\begin{cases}
a_j, & \text{ if } j<a_i;\\
a_{i+1}, & \text{ if } j=a_i;\\
a_{j-1}, & \text{ if } j>a_i.
\end{cases}
\]

\begin{remark}
The function $\gamma_n$ is well defined for all $n \geq 2$.  It is possible to see $\gamma(\bm{B})$ as the permutation $a_1 a_2 \cdots a_{a_i-1} a_{i+1} a_{a_i} a_{a_i+1} \cdots a_n$. It is important to notice that we are not only assigning the permutation $\gamma_n(\bm{B})$ a $\bm{B}$ but also the exceedance corresponding to the position $a_i$ of $\gamma_n(\bm{B})$. 
\end{remark}

\begin{example}
{From definition of  $\gamma_8$ we have:
\begin{align*}
\gamma_8(\{ 2\},\{ 5\},\{3,7 \},\{1 \},\{ 6\},\{ 8\},\{9 \})&=(2,5,\textbf{7},3,1,6,8,4) \ \text{and}\\
\gamma_8(\{ 5\},\{ 7\},\{3 \},\{1,2 \},\{ 6\},\{ 8\},\{9 \})&=(\textbf{2},5,7,3,1,6,8,4).
\end{align*}
Where we use the bold number to emphasize the exceedance that is in the image of the ordered partition.}
\end{example}

\begin{theorem}\label{thm:exceedances}
For $n\geq 2$, 
the function $\gamma_n:\UPFTypeOne_{n,\{n-1\}} \rightarrow \mathrm{Exc}_n$ is a bijection. In particular, $|\UPFTypeOne_{n,\{n-1\}}|=(n-1)\frac{n!}{2}$, the number of exceedances of the permutations of $[n]$.
\end{theorem}

\begin{proof}
     \textit{Injectivity:} Let $\bm{B}=(B_1,\dots,B_i,\dots,B_{n-1})=(\{a_1\},\dots,\{a_i,a_{i+1}\},\dots, \{a_n\})$ and $\mathbf{B'}=(B'_1,\dots,B'_i,\dots,B'_{n-1})=(\{a'_1\},\dots,\{a'_j,a'_{j+1}\},\dots, \{a'_n\})$ two partitions of $[n]$ such that $\gamma_n(\bm{B})=\gamma_n(\mathbf{B'})$. Then, 
        \[a_1 a_2 \cdots a_{a_i-1} a_{i+1} a_{a_i} a_{a_i+1} \cdots a_n = a'_1 a'_2 \cdots a'_{a'_j-1} a'_{j+1} a'_{a'_j} a'_{a'_j+1} \cdots a'_n
    \]
        Note that, since $\gamma_n(\bm{B})=\gamma_n(\mathbf{B'})$, we are looking at the same exceedance (i.e the same entry of the permutation). Hence, $a_i=a'_j$ and $a_{i+1}=a'_{j+1}$, this means both partitions have the same blocks, but may be not in the same order. However, by the definition of $\gamma_n$, since $a_i=a'_j$ and $a_1 a_2 \cdots a_{a_i-1} a_{i+1} a_{a_i} a_{a_i+1} \cdots a_n=a'_1 a'_2 \cdots a'_{a'_j-1} a'_{j+1} a'_{a'_j} a'_{a'_j+1} \cdots a'_n$, $i=j$ otherwise, the value $a_i$ would appear in two distinct indices of the permutation, which is a contradiction, and this implies that $\bm{B}=\bm{B'}$.

    \textit{Surjectivity: } Let $\sigma=\sigma_1\cdots\sigma_n$ be a permutation of $[n]$ with an exceedance in $j\in[n-1]$, that is $\sigma_j>j$. Define a partition of $[n]$, $\bm{B}=(B_1,\dots,B_{n-1})$ as follows:
        \begin{itemize}
    \item[$\bullet$] if $\sigma^{-1}_j<j$, then
    \[
    B_i:=
    \begin{cases}
    \{\sigma_i\}, & \textit{ if } i<j \textit{ and } i\neq \sigma^{-1}_j;\\
    \{\sigma_{i+1}\}, & \textit{ if } i\geq j;\\
    \{j,\sigma_j\}, & \textit{ if } i=\sigma^{-1}_j.
    \end{cases}
    \]
    \item[$\bullet$] if $\sigma^{-1}_j>j$ then
    \[
    B_i:=
    \begin{cases}
    \{\sigma_i\}, & \textit{ if } i<j;\\
    \{\sigma_{i+1}\}, & \textit{ if } i\geq j \textit{ and } i\neq \sigma^{-1}_j-1;\\
    \{j,\sigma_j\}, & \textit{ if } i=\sigma^{-1}_j-1.
    \end{cases}
    \]

    In both cases, by  the definition of $\gamma$, we have $\gamma_n(\bm{B})=\sigma$.
\end{itemize}
Hence, $\gamma_n$ is a bijection, and thus 
\[
|\UPFTypeOne_{n,\{n-1\}}|= (n-1) \frac{n!}{2},
\]
as desired.
\end{proof}

\section{Restricted Fubini rankings and Restricted Unit Interval Parking Functions of Type 2}\label{sec:restricted type 2}

\begin{definition}\label{def:restricted FR type 2}
We define $\FRTypeTwo_{n,S}$ as the set of Fubini rankings with $n$ competitors if, for each entry $x$ in $\alpha\in \FRTypeTwo_{n,S}$, the number of occurrences of $x$ in $\alpha$ belongs to $S$. 
The set $\FRTypeTwo_{n,S}$ is called the set of \emph{$S$-restricted Fubini rankings of type 2}.
\end{definition}

Intuitively, an element of  $\FRTypeTwo_{n,S}$ is a ranking where a tie among $k$ participants is allowed only if $k$ is an element of $S$. 
We give some initial results for restricted Fubini rankings of type 2.

\begin{lemma}
Let $n\geq 0$.
\begin{enumerate}
\item If $S=\{1\}$, then $\FRTypeTwo_{n,S}=\mathfrak{S}_n$, the set of pemrutations of $[n]$.

\item If $S=\{n\}$, then $\FRTypeTwo_{n,S}=\{(1,1,\ldots,1)\in[n]^n\}$.

\item If $S = \mathbb{Z}^+$, then $\FRTypeTwo_{n,S}=\FR_n$, and the e.g.f.\  of the counting sequence $\{ |\FRTypeTwo_{n,S} |\}_{n \geq 0} $  is given by $\frac{1}{2-e^x}$.

\item If $S=\{2k \ | \ k \in \mathbb{Z}^+ \}$, then  the e.g.f.\ of the counting sequence $\{ |\FRTypeTwo_{n,S}| \}_{n \geq 0} $ is $\frac{1}{2-\cosh{x}}$.

\item If $S=\{2k+1 \ | \ k \in \mathbb{Z}^+ \}$, then the e.g.f.\ of the counting sequence $\{ \FRTypeTwo_{n,S} \}_{n \geq 0} $ is $\frac{1}{1-\sinh{x}}$.
\end{enumerate}
\end{lemma}
\begin{proof}
We consider each case below.
\begin{enumerate}
\item If $S=\{1\}$, then $\FRTypeTwo_{n,S}$ consists of  all  permutations of $n$, as only one participant is allowed per position.

\item If $S=\{n\}$, then the only element of $\FRTypeTwo_{n,S}$ is the tuple of $n$ ones, since all $n$ participants must occupy the same position.

\item If $S = \mathbb{Z}^+$, then $\FRTypeTwo_{n,S}$ includes all elements of $\FR_n$, as ties of any size are permitted. In this case, $\FRTypeTwo_{n,S}$ is the set of all ordered set partitions of $[n]$, and the e.g.f.\  of the counting sequence $\{ |\FRTypeTwo_{n,S} |\}_{n \geq 0} $  is given by $\frac{1}{2-e^x}$.

\item If $S=\{2k \ | \ k \in \mathbb{Z}^+ \}$, then $\FRTypeTwo_{n,S}$ is the set of all ordered set partitions of $[n]$ only in blocks with an even number of elements, and the e.g.f.\  of the counting sequence $\{ |\FRTypeTwo_{n,S}| \}_{n \geq 0} $ is $\frac{1}{2-\cosh{x}}$.

\item If $S=\{2k+1 \ | \ k \in \mathbb{Z}^+ \}$, then $\FRTypeTwo_{n,S}$ is the set of all ordered set partitions of $[n]$ only in blocks with an odd number of elements, and the e.g.f.\  of the counting sequence $\{ \FRTypeTwo_{n,S} \}_{n \geq 0} $ is $\frac{1}{1-\sinh{x}}$.\qedhere
\end{enumerate}
\end{proof}
Next we give a formula for the number of $S$-restricted Fubini rankings of type 2 when $S$ is a singleton set.
\begin{theorem}\label{thm:FR type 2 k divides n}
If $S=\{k\}$ for some $k$ that divides $n$, then $|\FRTypeTwo_{n,S}|=\frac{n!}{(k!)^{n/k}}$.
\end{theorem}
\begin{proof}
We can arrange $n$ elements in $n!$ ways. Since each block must have exactly $k$ elements, we can split the permutation into consecutive blocks of size $k$. There are $k!$ ways to arrange elements within each block of size $k$, since there are $n/k$ such blocks, the total number of ways to arrange the elements of a permutation without changing content of the blocks is $(k!)^{n/k}$. By dividing by this number we eliminate the overcounting that arises from permuting the elements within each block, and thus the number of $\FRTypeTwo_{n,S}$ is $\frac{n!}{(k!)^{n/k}}$. 
\end{proof}

We now give a  formula for the number of $S$-restricted Fubini rankings of type 2 for any finite set $S$.
\begin{theorem}\label{thm:FR type 2}
Let  $S=\{s_1,\dots,s_j\}$ be a finite subset of positive integers integers, and let $C=\{\bm{c}=(c_1, \ldots, c_j) : \sum_{i=1}^j s_i c_i = n\}$. Then
$$\FRTypeTwo_{n,S}=n! \sum_{\bm{c} \in C} \binom{\sum_{k=1}^j c_{k}}{c_1,\dots,c_j} \prod_{i=1}^j (s_i)!^{-c_i}.$$
\end{theorem}
\begin{proof}
We can arrange $n$ elements in $n!$ ways. Given a partition tuple \(\bm{c} \in C\), we can split the permutation into $\sum_{k=1}^j c_{k}$ consecutive blocks. To assign sizes to these blocks according to the counts in \(\bm{c}\), we simply choose how to group the \(\sum_{k=1}^j c_{k}\) blocks into categories: \(c_1\) blocks of size \(s_1\), \(c_2\) of size \(s_2\), and so on. The number of ways to do this is given by the multinomial coefficient:
\[
\binom{\sum_{k=1}^j c_{k}}{c_1, c_2, \ldots, c_j}.
\]

Next, we account for internal arrangements within each block. Since a block of size $s_i$ can be arranged in $s_i!$ ways, and there are $c_i$ such blocks, we must divide by $(s_i!)^{c_i}$ to avoid overcounting. Thus, the number of corresponding ordered set partitions is  
$$n! \binom{\sum_{k=1}^j c_{k}}{c_1,\dots,c_j} \prod_{i=1}^j (s_i)!^{-c_i}.$$
Summing over all valid partition vectors $\bm{c} \in C$ gives the total count of $\FRTypeTwo_{n,S}$.  
\end{proof}

We have an analogous definition in the case of the unit interval parking functions.

\begin{definition}\label{def:restrict UPF type 2}
We define $\UPFTypeTwo_{n,S}$ as the set of unit interval parking functions of length $n$ such that the size of each block belongs to $S$, where $S$ is a subset of positive integers. If $\alpha \in \UPFTypeTwo_{n,S}$, we say that $\alpha$ is an \emph{$S$-restricted unit interval parking functions of type 2}.
\end{definition}

Analogously to the argument given in \Cref{biyection}, but considering the size of the blocks and the number of ties in each position instead of the number of blocks and the number of different positions, we can state the following theorem.

\begin{theorem}\label{thm:FR type 2 bijection UPF type 2}
For all $n \in \Z^+$ and any set $S$ of positive integers, the set $\FRTypeTwo_{n,S}$ is in bijection with $ \UPFTypeTwo_{n,S}$.
\end{theorem}

\section{Restricted Fubini rankings and Restricted Unit Interval Parking Functions of Type 3}\label{sec:restricted type 3}

In this section, we introduce a different way of restricting Fubini rankings and unit interval parking functions. 

We now impose restrictions on each \textit{rank} of the Fubini ranking according to the entries of a given sequence of positive integers. We make this definition precise next.

\begin{definition}\label{def: FR position vector}
Let $ \alpha \in \FR_n $  have $k$ distinct ranks. The \emph{position vector} of $\alpha$ is the vector $\mathbf{c}=(c_1,\ldots, c_k)$, where $c_i$ is the number of entries of $\alpha$ that take the value $c_0+c_1+\cdots+c_{i-1}$  for each $i=1,\dots,k$, with $c_0=1$. 
\end{definition}
We illustrate \Cref{def: FR position vector} with the following.
\begin{example}
Consider $\alpha=(4,2,1,2,4,7,4)$. It is straightforward to check that $\alpha \in \FR_7$. Then, given $c_0=1$, we have that $c_1$ the number of occurrences of $c_0=1$ in $\alpha$, this is $c_1=1$. Now, $c_2$ is the number of occurrences of $c_0+c_1=2$ in $\alpha$, then $c_2=2$. Following this procedure we find that $c_3=3$ and $c_4=1$, and then the position vector of $\alpha$ is $\bm{c}=(1,2,3,1)$.
\end{example}  
Next we define the set of $S$-restricted Fubini rankings of type 3.
\begin{definition}\label{def: restricted FR type 3}
Let $ \alpha \in \FR_n $  have position vector $\bm{c}=(c_1,\dots,c_k)$. A sequence of positive integers  $S=(s_1,s_2,\dots)$ is called a \textit{block sequence} for $\alpha$ if $c_i \leq s_i$ for each $i=1,2, \ldots,k$. We denote by $\FRTypeThree_{n,S}$ the set of Fubini rankings of length $n$ with block sequence $S$. If $\alpha \in \FRTypeThree_{n,S}$, we say that $\alpha$ is \emph{$S$-restricted Fubini rankings of type 3}.
\end{definition}

We illustrate \Cref{def: restricted FR type 3} with the following.
\begin{example}
Let $S=(2,4,6,\dots)$ be the sequence of even positive integers. Consider the Fubini ranking $\alpha= (1,3,1,6,3,3,6)$  of length 7, which has position vector $\bm{c}=(2,3,2)$. Checking the conditions of \Cref{def: restricted FR type 3}, we verify that  
\[
c_1 = 2 \leq 2 = s_1, \quad c_2 = 3 \leq 4 = s_2, \quad \text{and} \quad c_3 = 2 \leq 6 = s_3.
\]
Thus, we conclude that $\alpha \in \FRTypeThree_{7,S} $. 

On the other hand, consider the Fubini ranking $\beta = (1,1,1,4)$ of length 4, which has position vector $\bm{c} = (3,1)$.  Since
\[
c_1 = 3 > 2 = s_1,
\]  
we see that $\beta \notin \FRTypeThree_{4,S}$.
 \end{example}

We now give a formula for the number of $S$-restricted Fubini rankings of type 3.

\begin{theorem}\label{counting:type3}
For all $n \in \Z^+$ and every sequence of positive integers $S=(s_1,s_2,\ldots)$, we have 
\begin{align*}
|\FRTypeThree_{n,S} |= \sum_{k=1}^{n} \sum_{(c_1,c_2,\ldots,c_k) \in M_{S,k,n}} \binom{n}{c_1,c_2,\ldots ,c_k}, 
\end{align*}
where 
$
M_{S,k,n} = \left\{(c_1, c_2, \dots, c_k)  \mid  c_1+c_2+\cdots + c_k= n  \text { and } c_i\in[s_i] \text{ for all } 1\leq i \leq k\right\}$.
\end{theorem}

\begin{proof}
Let $\alpha \in \FRTypeThree_{n,S}$  be an $S$-restricted Fubini rankings of type 3 with position vector $\bm{c}=(c_1,c_2,\ldots,c_k)$. By definition, $\bm{c}$ must be an element of $M_{S,k,n}$, ensuring that $\alpha$ has exactly $k$ distinct ranks and satisfies the block sequence $S$.   To construct such an $\alpha$, we need to assign its $n$ entries to the values determined by $\bm{c}$. The number of ways to distribute these $n$ elements into blocks of sizes $c_1, c_2, \dots, c_k$ is given by the multinomial coefficient $\binom{n}{c_1,c_2,\ldots,c_k}$.  
Finally, summing over all possible values of $k$, we obtain the desired result.
\end{proof}

Another interesting way to derive \Cref{counting:type3} is through exponential generating functions. Let $\bm{c}=(c_1,c_2,\ldots,c_k) \in M_{S,k,n}$. Is known that the exponential generating function of a set of size $c_i$ is given by $\frac{x^{c_i}}{c_i!}$ for each $i=1,\dots,k$. Taking the product of these functions and summing over all elements of $M_{S,k,n}$ and all possible values of $k$, we obtain  the exponential generating function for the number of elements in $\FRTypeThree_{n,S}$: 
\begin{align*}
\sum_{n \geq 1 }\sum_{k=1}^n \sum_{\bm{c} \in M_{S,k,n}} \frac{x^{c_1}}{c_1!} \frac{x^{c_2}}{c_2!} \cdots \frac{x^{c_k}}{c_k!} 
&=  \sum_{n \geq 1 } \sum_{k=1}^n \sum_{\bm{c} \in M_{S,k,n}} \frac{x^n}{c_1! c_2! \cdots c_k!} \\
&= \sum_{n \geq 1 } \sum_{k=1}^n \sum_{\bm{c} \in M_{S,k,n}} \binom{n}{c_1, c_2, \ldots, c_k} \frac{x^n}{n!}.
\end{align*}
Thus, we have proved the theorem once again.

We apply \Cref{counting:type3} next.
\begin{example}
Let $n=7$ and let $S$ be the sequence of even positive integers. 
Consider the element $\bm{c}=(2,1,3,1)\in M_{S,4,7}$.
We have shown that there are $\binom{7}{2,1,3,1}$ elements in $\UPFTypeThree_{7,S}$ associated with $\bm{c}$,  each of which can be constructed by forming an ordered set partition of $[7]$ into blocks of sizes given by $\bm{c}$. 
This follows from the fact that the enumeration of Fubini rankings is given by the Fubini numbers.  For example, $(\{ 3,4\}, \{ 5\}, \{ 1, 6,7 \}, \{ 2 \})$ corresponds to the Fubini ranking $(4,7,1,1,3,4,4)$.
\end{example}
Next we define the set of $S$-restricted unit interval parking functions of type 3.

\begin{definition}\label{def:restricted UPF type 3}
Let $\beta \in \UPF_n$  have block structure $\beta' = \pi_1 | \pi_2 | \cdots | \pi_k$. A sequence $S=(s_1,s_2,\dots)$ of positive integers is called a \textit{block sequence} for $\beta$ if $|\pi_i| \leq s_i$ for each $i=1,\ldots,k$. We denote by $\UPFTypeThree_{n,S}$ the set of unit interval parking functions of length $n$ with block sequence $S$. If $\beta \in \UPF_n$, we say that $\beta$ is an \emph{$S$-restricted unit interval parking functions of type 3}.
\end{definition}
 Before
illustrating \Cref{def:restricted UPF type 3} we recall the following characterization of parking functions. Let $\alpha=(a_1,a_2,\ldots,a_n)\in[n]^n$. The tuple $\alpha$ is a parking function if and only if the nondecreasing rearrangement of its entries, denoted $\alpha^\uparrow=(a_1',a_2',\ldots,a_n')$ satisfies $a_i'\leq i$ for all $1\leq i\leq n$.

\begin{example}\label{ex1:type3}
Let $S=(4,5,6,7,8,\ldots)$ and let $\beta = (5,1,5,7,1,2,9,9,3,10,7)$. 
Observe that the nondecreasing rearrangement of $\beta$ is given by $\beta^{\uparrow}=(1,1,2,3,5,5,7,7,9,9,10)$. It is easy to see that every entry of $\beta^{\uparrow}$ satisfies the condition of being less than or equal to its position.  
This implies that $\beta \in \PF_{10}$. 
Now, the block structure of $\beta^{\uparrow}$ is $\beta'= 1123 |55 | 77 | 9 9 (10)$, meaning that $\beta$ respects the relative order given by $\beta'$. Noticing that $\beta'$ is a unit interval parking function, this is a sufficient condition for $\beta$ to belong  to $\UPF_{10}$. Finally, the block sizes in $\beta'$ form the tuple $(4,2,2,3)$, which satisfies the given restrictions.  Thus, we  conclude that $\beta \in \UPFTypeThree_{10,S}$.
\end{example}

Using the same idea as in \Cref{biyection}, one can easily verify that for all $n \in \Z^+$ and any sequence $S$ of positive integers, there exists a bijection between $\FRTypeThree_{n,S}$ and $\UPFTypeThree_{n,S}$. This correspondence follows directly by applying the injective functions \( \phi \) and \( \psi \), defined in \eqref{phi} and \eqref{psi}, respectively, to the block structure of a unit interval parking function and the list of positions in the associated Fubini ranking. We state this result next and follow that with and illustration of the bijection.

\begin{theorem}\label{thm:UPF type 3 bijection FR type 3}
For all $n \in \Z^+$ and any sequence $S$ of positive integers, the set $\FRTypeThree_{n,S}$ is in bijection with $\UPFTypeThree_{n,S}$.
\end{theorem}

\begin{example}\label{ex2:type3}
Consider the tuple $\alpha=(3,6,3,6,5,6,1,1,9)\in[9]^9$, which one can confirm is a Fubini ranking of length 9. However, this is evident since we can associate $\alpha$ with the ordered set partition of $[9]$ given by $(\{7,8\}, \{ 1,3\}, \{5\},  \{6,2,4\}, \{9\})$. 
Applying the function $\phi$ defined in \eqref{phi}, we obtain  $\phi(\alpha)=(3,6,3,6,5,7,1,1,9)$, which is a unit interval parking function of length 9.

Now, let $S=(2,2,3,3,4,4,\ldots)$. Since the number of entries at each position in
$\alpha$ satisfies the restrictions given by \( S \), we conclude that $\alpha \in \FRTypeThree_{9,S}$. On the other hand, the block structure of $\phi(\alpha)$ is given by $\phi(\alpha)'=11 | 33 | 5 |667| 9$, which satisfies the same restrictions.  Thus, $\phi(\alpha) \in \UPFTypeThree_{9,S}$, as expected.
\end{example}

The next result is simple, but it plays a key role in our final results.

\begin{proposition}\label{reag:type3}
For all $n \in \Z^+$, any rearrangement $\sigma$ of $\alpha \in  \UPFTypeThree_{n,S}$ that remains a unit interval parking function also satisfies  $\sigma \in  \UPFTypeThree_{n,S}$.
\end{proposition}
\begin{proof}
We already know that if a rearrangement $\sigma$ of a unit parking function $\alpha$ is still a parking function, then $\sigma$ must preserve the same block structure as $\alpha$. As  $\alpha$ is an element of $\UPFTypeThree_{n,S} $, then $\sigma$ is also an element of this set.
\end{proof}
Now, let us focus on the special case of elements in $\UPFTypeThree_{n,S}$ where the block sizes in the parking function structure precisely match the given block sequence.

\begin{definition}\label{def:fully determined}
Let $S=(s_1,s_2,\dots)$ be an arbitrary sequence of positive integers, and let $\beta \in \UPFTypeThree_{n,S}$ with block structure $\beta'= \pi_1 | \pi_2 | \cdots | \pi_k$. We say that $\beta$ is \emph{fully determined} by $S$ if $|\pi_i|=s_i$ for each $i=1,\dots,k$.
\end{definition}

We illustrate \Cref{def:fully determined} next.
\begin{example}
Consider  $\beta=(5,1,5,7,1,2,9,9,3,10,7)$ from \Cref{ex1:type3}. Then, $\beta$ is fully determined by any sequence $S$ whose initial initial entries are $(4,2,2,3)$. 
Similarly, take $\phi(\alpha)=(3,6,3,6,5,7,1,1,9)$ from \Cref{ex2:type3}. In this case, $\phi(\alpha)$ is fully determined by any sequence $S$ whose initial initial entries are $(2,2,1,3,1)$.
\end{example}

The following result follows directly from \Cref{counting:type3}. 
\begin{corollary}\label{counting2:type3}
For all $n \in \Z^+$ and any arbitrary sequence of positive integers $S=(s_1,s_2,\dots)$, there are $\binom{n}{s_1,s_2,\ldots,s_k}$ elements in $\UPFTypeThree_{n,S}$ that are fully determined by $S$, provided that there exists an integer $k \in \Z^+$  for which $s_1+s_2+\cdots+s_k=n$.
\end{corollary}

\Cref{T1} states that a unit interval parking function $\beta$ of length $n$ with block structure $\beta' = \pi_1 | \pi_2 | \cdots | \pi_k$ has $\binom{n}{|\pi_1|,|\pi_2|,\ldots,|\pi_k|}$ rearrangements that are also unit interval parking functions. Then, by \Cref{reag:type3}, the rearrangements of a unit interval parking function fully determined by $S$ that are also unit interval parking functions are precisely all unit interval parking functions fully determined by $S$.

\section{Future Work}\label{sec:future}

Much of our current work relies on a bijection between Fubini rankings and ordered set partitions. 
An intriguing question is how these constructions behave under $q$-analogues. For instance, understanding the enumerative and algebraic properties of $q$-analogues of Fubini rankings or $S$-restricted unit interval parking functions may connect with the theory developed by by Ishikawa, Kasraoui, and Zeng in \cite{istvawa2006statisticsorderedpartitionssets}.

\section{Acknowledgments}
This paper was developed during \textit{Beyond Research 2024-II: Combinatorics of Parking Functions}, an undergraduate research experience hosted by the Universidad Nacional de Colombia and the University of Wisconsin--Milwaukee.
P.~E.~Harris  supported in part through a grant from the Simons Foundation (Travel Support for Mathematicians).

\bibliographystyle{plain}
\bibliography{bibliography.bib}

\end{document}